\documentclass[12pt, reqno]{amsart}
\usepackage{amssymb, amsthm, fullpage}

\newtheorem{thm}{Theorem}[section]

\newtheorem{lemma}[thm]{Lemma}
\newtheorem{theorem}[thm]{Theorem}

\numberwithin{equation}{section}

\theoremstyle{definition}
\newtheorem{remark}[thm]{Remark}

\renewcommand{\qed}{\rule{3mm}{3mm}}
\renewenvironment{proof}
    {\vspace{1mm}\noindent\textbf{Proof.}}
    {\hspace*{\fill} $\qed$\vspace{1mm}}

\newcommand{\bR}{\mathbb R}
\renewcommand{\Bbb}{\mathbb}
\DeclareMathOperator{\ddiv}{div}

\usepackage{xcolor}

\begin{document}
\title[Monotonic convergence for quasilinear elliptic systems]{Monotonic convergence of
positive radial solutions for general quasilinear elliptic systems}
\author{Daniel Devine}
\author{Paschalis Karageorgis}
\address{School of Mathematics, Trinity College Dublin, Ireland.}
\email{dadevine@tcd.ie} \email{pete@maths.tcd.ie}


\begin{abstract}
We study the asymptotic behavior of positive radial solutions for quasilinear elliptic systems that
have the form
\begin{equation*}
\left\{ \begin{aligned}
\Delta_p u &= c_1|x|^{m_1} \cdot g_1(v) \cdot |\nabla u|^{\alpha} &\quad\mbox{ in } \mathbb R^n,\\
\Delta_p v &= c_2|x|^{m_2} \cdot g_2(v) \cdot g_3(|\nabla u|) &\quad\mbox{ in } \mathbb R^n,
\end{aligned} \right.
\end{equation*}
where $\Delta_p$ denotes the $p$-Laplace operator, $p>1$, $n\geq 2$, $c_1,c_2>0$ and $m_1, m_2, \alpha
\geq 0$.  For a general class of functions $g_j$ which grow polynomially, we show that every
non-constant positive radial solution $(u,v)$ asymptotically approaches $(u_0,v_0) = (C_\lambda
|x|^\lambda, C_\mu |x|^\mu)$ for some parameters $\lambda,\mu, C_\lambda, C_\mu>0$.  In fact, the
convergence is monotonic in the sense that both $u/u_0$ and $v/v_0$ are decreasing.  We also obtain
similar results for more general systems.
\end{abstract}

\maketitle

\section{Introduction}
We study the positive radial solutions of the quasilinear elliptic system
\begin{equation} \label{sys|1}
\left\{ \begin{aligned}
\Delta_p u &= c_1|x|^{m_1} \cdot g_1(v) \cdot |\nabla u|^{\alpha} &\quad\mbox{ in } \Omega, \\
\Delta_p v &= c_2|x|^{m_2} \cdot g_2(v) \cdot g_3(|\nabla u|) &\quad\mbox{ in } \Omega,
\end{aligned} \right.
\end{equation}
where $\Delta_p u = \ddiv(|\nabla u|^{p-2}\nabla u)$ denotes the $p$-Laplace operator, $p>1$, $n\geq 2$,
$c_1,c_2>0$ and $m_1, m_2, \alpha \geq 0$.  We are mainly concerned with the asymptotic behavior of
global solutions in the case $\Omega = \mathbb R^n$, but we shall also consider local solutions over the
open ball $\Omega = B_R$ around the origin.  For the functions $g_j$, we impose some generic conditions
that we outline below.  These allow each $g_j$ to be an arbitrary polynomial with non-negative
coefficients.

Semilinear elliptic systems that do not involve any gradient terms have been extensively studied in the
past \cite{BVP2001, CR2002, GMRZ2009, L2011, LW1999, LW2000} and the same is true for quasilinear
elliptic systems without gradient terms \cite{ACM2002,BV2000, BVG2010, BVG1999, BVP2001, CFMT2000,
F2011}.  On the other hand, systems such as \eqref{sys|1} have also attracted interest in recent years
\cite{DLS2005, F2013, FV2017, GGS2019, S2015}.  Their study was initiated by D\'iaz, Lazzo and Schmidt
\cite{DLS2005} who concentrated on the special case
\begin{equation*}\label{sysDLS}
\left\{ \begin{aligned}
\Delta u &=v &&\quad\mbox{ in } \Omega,\\
\Delta v &=|\nabla u|^2 &&\quad\mbox{ in } \Omega.
\end{aligned} \right.
\end{equation*}
This system arises through a prototype model for a viscous, heat-conducting fluid, and its solutions
correspond to steady states of a related parabolic system that describes the unidirectional flow of the
fluid.

More recently, Singh \cite{S2015} considered general semilinear systems of the form
\begin{equation*}\label{sysG}
\left\{ \begin{aligned}
\Delta u &= v^m &&\quad\mbox{ in } \Omega,\\
\Delta v &= g(|\nabla u|)&&\quad\mbox{ in } \Omega,
\end{aligned} \right.
\end{equation*}
where $m>0$ and $g\in C^1([0,\infty))$ is non-decreasing and positive in $(0,\infty)$. He found optimal
conditions for the existence of positive radial solutions in either $\Omega=B_R$ or $\Omega = \bR^n$,
and he determined the asymptotic behavior of global solutions, assuming that $g(s) = s^k$ for some
$k\geq 1$ and that $k,m$ satisfy some additional hypotheses.

The existence results of \cite{S2015} were extended by Filippucci and Vinti \cite{FV2017} to the
$p$-Laplace operator for any $p\geq 2$.  Ghergu, Giacomoni and Singh \cite{GGS2019} then studied the
case
\begin{equation} \label{sysGGS}
\left\{ \begin{aligned}
\Delta_{p} u &= v^{k_1} \cdot |\nabla u|^{\alpha} &\quad\mbox{ in } \Omega, \\
\Delta_{p} v &= v^{k_2} \cdot |\nabla u|^{k_3} &\quad\mbox{ in } \Omega,
\end{aligned} \right.
\end{equation}
where $k_1,k_2,k_3,\alpha \geq 0$.  Once again, optimal results were established for the existence of
positive radial solutions, but the asymptotic behavior of global solutions was now settled under the
very natural assumptions
\begin{equation} \label{excon}
0\leq \alpha < p-1, \qquad (p-1-\alpha)(p-1-k_2)>k_1k_3
\end{equation}
which are related to the existence of solutions.  Assuming that \eqref{excon} holds, in particular, one
can easily check that the system \eqref{sysGGS} has an explicit solution of the form
\begin{equation} \label{special}
u_0(r) = C_\lambda r^\lambda = C_\lambda |x|^\lambda, \qquad
v_0(r) = C_\mu r^\mu = C_\mu |x|^\mu
\end{equation}
for some $\lambda,\mu, C_\lambda, C_\mu>0$.  According to \cite[Theorem 2.3]{GGS2019}, all non-constant
positive radial solutions of \eqref{sysGGS} in $\Omega = \bR^n$ have the same behavior at infinity, and
they satisfy
\begin{equation} \label{conv}
\lim_{r\to \infty} \frac{u(r)}{u_0(r)} = \lim_{r\to \infty} \frac{v(r)}{v_0(r)} = 1.
\end{equation}

Our first goal in this paper is to further extend \cite[Theorem 2.3]{GGS2019} to a much more general
class of quasilinear elliptic systems.  Instead of the system \eqref{sysGGS} that was considered in
\cite{GGS2019}, we shall thus study the system \eqref{sys|1} which contains additional factors and more
general nonlinearities. When it comes to the functions $g_j$, we shall assume that
\begin{itemize}
\item[(A1)]
each $g_j$ is differentiable and non-decreasing in $[0,\infty)$;
\item[(A2)]
each $g_j$ is positive in $(0,\infty)$ and $g_1,g_2$ are increasing in $(0,\infty)$;
\item[(A3)]
there exist constants $k_j\geq 0$ such that $g_j(s)/s^{k_j}$ is non-increasing in $(0,\infty)$ with
\begin{equation} \label{limit}
\lim_{s\to \infty} \frac{g_j(s)}{s^{k_j}} = 1 \qquad \text{for each $j=1,2,3$.}
\end{equation}
\end{itemize}
Needless to say, system \eqref{sysGGS} arises when $g_j(s)= s^{k_j}$ for each $j$, but our assumptions
allow much more general functions.  For instance, each $g_j(s)$ could be any sum of non-negative powers
of $s$ with non-negative coefficients, as long as the highest power is $s^{k_j}$. In all these cases, we
show that non-constant positive radial solutions of \eqref{sys|1} in $\Omega = \mathbb R^n$ exhibit the
same asymptotic behavior as the corresponding solutions of the limiting system
\begin{equation} \label{sys|2}
\left\{ \begin{aligned}
\Delta_p u &= c_1|x|^{m_1} \cdot v^{k_1} \cdot |\nabla u|^{\alpha} &\quad\mbox{ in } \Omega, \\
\Delta_p v &= c_2|x|^{m_2} \cdot v^{k_2} \cdot |\nabla u|^{k_3} &\quad\mbox{ in } \Omega,
\end{aligned} \right.
\end{equation}
and that they all approach $(u_0,v_0)= (C_\lambda r^\lambda, C_\mu r^\mu)$ for some parameters
$\lambda,\mu, C_\lambda, C_\mu>0$.

Our second goal in this paper is to show that the convergence \eqref{conv} actually occurs in a
monotonic fashion. In fact, we shall prove that both $u/u_0$ and $v/v_0$ are decreasing for any positive
radial solution $(u,v)$ which satisfies \eqref{sys|1} in $\Omega = B_R$.  This observation appears to be
entirely new, in any case whatsoever, and it is a key ingredient in our approach.

When it comes to the results of \cite{GGS2019, S2015}, the authors established \eqref{conv} using a
nonlinear change of variables and the theory of cooperative dynamical systems.  In this paper, we
introduce a new linear change of variables which seems more natural, and we apply simple continuity
arguments to prove the monotonicity of $u/u_0$ and $v/v_0$ directly.

\begin{thm} \label{MainThm}
Assume (A1)-(A3) and \eqref{excon}.  Let $(u,v)$ be a non-constant positive radial solution of the
system \eqref{sys|1}, where $n\geq 2$, $m_1, m_2 \geq 0$ and $c_1, c_2 > 0$. Then the associated system
\eqref{sys|2} has a unique solution of the form $(u_0,v_0) = (C_\lambda r^\lambda, C_\mu r^\mu)$, where
each of the parameters $\lambda, \mu, C_\lambda, C_\mu$ is positive.

\begin{itemize}
\item[(a)]
If $\Omega = B_R$, then $\frac{u(r)}{u_0(r)}$, $\frac{v(r)}{v_0(r)}$, $\frac{u'(r)}{u_0'(r)}$,
$\frac{v'(r)}{v_0'(r)}$ are all decreasing for each $0<r<R$.

\item[(b)]
If $\Omega = \mathbb R^n$, then $(u,v)$ and $(u_0,v_0)$ have the same behavior at infinity, namely
\begin{equation} \label{asym}
\lim_{r\to \infty} \frac{u(r)}{u_0(r)} = \lim_{r\to \infty} \frac{v(r)}{v_0(r)} =
\lim_{r\to \infty} \frac{u'(r)}{u_0'(r)} = \lim_{r\to \infty} \frac{v'(r)}{v_0'(r)} = 1.
\end{equation}
\end{itemize}
\end{thm}

\begin{remark}
Although we have mainly focused on the system \eqref{sys|1}, our method of proof applies verbatim for
several variations of that system.  When it comes to the first equation, for instance, one may replace
the factor $g_1(v)$ by $g_1(v) \cdot g_4(u)$, where $g_4$ is also subject to the assumptions (A1)-(A3).
When it comes to the second equation, one may similarly replace $g_2(v)$ by $g_2(v)\cdot g_5(u)$, and
also include the factor $|\nabla v|^\beta$, where $\beta\geq 0$.  All these variations can be treated
using the same approach by merely adjusting the existence condition \eqref{excon}.  For the sake of
simplicity, however, we shall not bother to treat them explicitly.
\end{remark}

As we show in Sections \ref{Sec3} and \ref{Sec4}, Theorem \ref{MainThm} is actually valid for more
general systems obtained by replacing the coefficients $c_1|x|^{m_1}$, $c_2|x|^{m_2}$ in \eqref{sys|1}
by arbitrary non-decreasing functions $f_1(|x|)$, $f_2(|x|)$.  In that case, however, the asymptotic
profile $(u_0,v_0)$ does not have the form $(C_\lambda |x|^\lambda, C_\mu |x|^\mu)$, and it is not
necessarily an elementary function.  We shall mostly concentrate on the general case \eqref{Sys1} and
then deduce Theorem \ref{MainThm} as a special case.

To some extent, Theorem \ref{MainThm} is reminiscent of Wang's classical result \cite{W1993} regarding
the positive radial solutions of the scalar equation $-\Delta u = u^p$, where $p>1$.  This equation has
decaying solutions of the form $u_0 = C_\lambda |x|^\lambda$, where $\lambda=-2/(p-1)$, and every
positive radial solution $u$ behaves like $u_0$ at infinity.  However, the convergence $u/u_0\to 1$ is
monotonic only for supercritical powers \cite{KS2007, W1993}.  A similar result holds for the scalar
equation $\Delta^2 u = u^p$ which has decaying solutions of the form $u_0 = C_\lambda |x|^\lambda$ with
$\lambda= -4/(p-1)$. Once again, every positive radial solution $u$ behaves like $u_0$ at infinity, but
the convergence $u/u_0\to 1$ is monotonic only for supercritical powers $p$; see \cite{FGK2009, K2009}
for more details.  It seems quite interesting that similar results can also be obtained for systems.

The remaining sections are organized as follows. In Section \ref{Sec2}, we introduce our general system
\eqref{Sys1}, and we show that all positive radial solutions are increasing and convex. In Section
\ref{Sec3}, we prove a new Monotonic Comparison Theorem which is a natural extension of Theorem
\ref{MainThm}(a). In Section \ref{Sec4}, we prove a new Asymptotic Comparison Theorem which similarly
extends Theorem \ref{MainThm}(b).  The proof of Theorem \ref{MainThm} is then given in Section
\ref{Sec5}.

\section{Monotonicity and convexity of solutions} \label{Sec2}
In this section, we study the positive radial solutions of the general system
\begin{equation}\label{Sys1}
\left\{ \begin{aligned}
\Delta_p u &= f_1(|x|) \cdot g_1(v) \cdot |\nabla u|^{\alpha}
&\quad\mbox{ in } \Omega,\\
\Delta_p v &= f_2(|x|) \cdot g_2(v) \cdot g_3(|\nabla u|)
&\quad\mbox{ in } \Omega,
\end{aligned} \right.
\end{equation}
where $\Omega = B_R$ and $\alpha\geq 0$, for any functions $f_i, g_j$ that satisfy the following
hypothesis:

\begin{itemize}
\item[(B1)]
each $f_i, g_j$ is continuous, non-decreasing in $[0,\infty)$ and positive in $(0,\infty)$.
\end{itemize}

\noindent Our main goal is to show that all positive radial solutions have components $u,v$ which are
increasing and convex.  This observation will play a crucial role in our subsequent analysis.  Although
our proofs are very similar to the ones presented in \cite{GGS2019, S2015}, our system \eqref{Sys1} is
more general, so we also include the proofs for the sake of completeness.

\begin{lemma}[Monotonicity] \label{Mono}
Assume (B1), $n\geq 2$ and $\alpha \geq 0$.  Suppose that $(u,v)$ is a non-constant radial solution of
\eqref{Sys1} in $\Omega = B_R$ which is positive for all $0<r<R$.  Then one has $u'(r)>0$ and $v'(r)>0$
for all $0<r<R$.
\end{lemma}

\begin{proof}
Since $(u,v)$ is a radial solution of \eqref{Sys1}, it must satisfy the system
\begin{equation} \label{sysrad1}
\left\{ \begin{aligned}
&\left[ r^{n-1} u'(r) |u'(r)|^{p-2} \right]' = r^{n-1} f_1(r) g_1(v(r)) \cdot |u'(r)|^{\alpha}
&&\text{for all $0<r<R$,} \\
&\left[ r^{n-1} v'(r) |v'(r)|^{p-2} \right]' = r^{n-1} f_2(r) g_2(v(r)) \cdot g_3(|u'(r)|)
&&\text{for all $0<r<R$,} \\
&u'(0)=v'(0)=0, \: u(r)>0, v(r)>0
&&\text{for all $0<r<R$.}
\end{aligned} \right.
\end{equation}
Once we integrate the first equation and recall our assumption (B1), we find that
\begin{equation*}
r^{n-1} u'(r) |u'(r)|^{p-2} = \int_0^r s^{n-1} f_1(s) g_1(v(s))\cdot |u'(s)|^{\alpha} \,ds > 0
\end{equation*}
for all $0<r<R$.  This gives $u'(r)>0$, and one may similarly argue that $v'(r)>0$.
\end{proof}

\begin{lemma}[Nonexistence] \label{Nonex}
Assume (B1), $n\geq 2$ and $\alpha \geq p-1$.  Then \eqref{Sys1} does not have any non-constant positive
radial solutions in $\Omega = B_R$ for any $R>0$.
\end{lemma}

\begin{proof}
Suppose $(u,v)$ is such a solution.  By Lemma \ref{Mono}, we must then have $u'(r)>0$ for all $0< r< R$.
Thus, the first equation in \eqref{sysrad1} can be expressed in the form
\begin{equation*}
\frac{\left[ r^{n-1} u'(r)^{p-1} \right]'}{r^{n-1}u'(r)^{p-1}} =
f_1(r) g_1(v(r)) \cdot u'(r)^{\alpha-p+1}
\end{equation*}
for all $0<r<R$.  Integrating this equation over the interval $[0,R/2]$ now gives
\begin{equation*}
\ln\left[ r^{n-1} u'(r)^{p-1} \right] \Big|_0^{R/2} =
\int_0^{R/2} f_1(s) g_1(v(s)) \cdot u'(s)^{\alpha-p+1} \,ds.
\end{equation*}
Since $\alpha\geq p-1$ and $f_1, g_1$ are continuous, the right hand side is obviously finite.  This is
a contradiction because $u'(0)=0$ and $p>1$, so the left hand side is infinite.
\end{proof}

\begin{lemma}[Convexity] \label{Conv}
Assume (B1), $n\geq 2$ and $0\leq \alpha < p-1$.  If $(u,v)$ is a non-constant radial solution of
\eqref{Sys1} in $\Omega = B_R$ which is positive for all $0<r<R$, then
\begin{align}
\frac{p-1-\alpha}{n(p-1-\alpha)+\alpha} \,f_1(r) g_1(v)
&\leq \left[u'(r)^{p-1-\alpha}\right]'
\leq \frac{p-1-\alpha}{p-1} \,f_1(r)g_1(v), \label{est1} \\
\frac{1}{n} \,f_2(r) g_2(v) g_3(u') &\leq \left[v'(r)^{p-1}\right]'
\leq f_2(r) g_2(v) g_3(u') \label{est2}
\end{align}
for all $0<r<R$.  In particular, both $u(r)$ and $v(r)$ are convex for all $0< r < R$.
\end{lemma}

\begin{proof}
Using Lemma \ref{Mono}, we see that the system \eqref{sysrad1} can be written as
\begin{equation*}
\left\{ \begin{aligned}
&\left[u'(r)^{p-1}\right]'+\frac{n-1}{r} \,u'(r)^{p-1} = f_1(r) g_1(v(r)) \cdot u'(r)^{\alpha}
&&\text{for all $0<r<R$,} \\
&\left[v'(r)^{p-1}\right]'+\frac{n-1}{r} \,v'(r)^{p-1} = f_2(r) g_2(v(r)) \cdot g_3(u'(r))
&&\text{for all $0<r<R$,} \\
&u'(0)=v'(0)=0, \: u(r)>0, v(r)>0
&&\text{for all $0<r<R$.}
\end{aligned} \right.
\end{equation*}
Rearranging the terms in the first equation now leads to the system
\begin{equation}\label{sysrad2}
\left\{ \begin{aligned}
&\left[ u'(r)^{p-1-\alpha} \right]' + \frac{\delta}{r} \,u'(r)^{p-1-\alpha} =
\frac{\delta}{n-1} \,f_1(r) g_1(v(r))
&&\text{for all $0<r<R$,} \\
&\left[v'(r)^{p-1}\right]'+\frac{n-1}{r} \,v'(r)^{p-1} = f_2(r) g_2(v(r)) \cdot g_3(u'(r))
&&\text{for all $0<r<R$,}
\end{aligned} \right.
\end{equation}
where $\delta$ is a positive constant that is defined by
\begin{equation} \label{delta}
\delta = \frac{(n-1) (p-1-\alpha)}{p-1} > 0.
\end{equation}
Since $u'(r)>0$ and $v'(r)>0$ for all $0<r<R$ by Lemma \ref{Mono}, it is then clear from \eqref{sysrad2}
that the upper bounds in \eqref{est1} and \eqref{est2} hold.  To prove the corresponding lower bounds,
we first rearrange terms and express \eqref{sysrad2} in the equivalent form
\begin{equation}\label{sysrad3}
\left\{ \begin{aligned}
&\left[ r^\delta u'(r)^{p-1-\alpha} \right]' = \frac{\delta}{n-1} \,r^\delta f_1(r) g_1(v(r))
&&\text{for all $0<r<R$,} \\
&\left[ r^{n-1} v'(r)^{p-1} \right]' = r^{n-1} f_2(r) g_2(v(r)) \cdot g_3(u'(r))
&&\text{for all $0<r<R$.}
\end{aligned} \right.
\end{equation}
Since $f_1,g_1$ are non-decreasing by our assumption (B1) and $v$ is increasing by Lemma \ref{Mono}, it
follows by the first equation of this system that
\begin{equation} \label{bd1}
r^\delta u'(r)^{p-1-\alpha}
= \frac{\delta}{n-1} \int_0^r s^\delta f_1(s) g_1(v(s)) \,ds
\leq \frac{\delta}{n-1} \,f_1(r)g_1(v(r)) \int_0^r s^\delta \,ds
\end{equation}
for all $0< r < R$.  In other words, one has the estimate
\begin{equation} \label{bound}
\frac{1}{r} \,u'(r)^{p-1-\alpha} \leq \frac{\delta}{(n-1)(\delta+1)} \, f_1(r)g_1(v(r))
\end{equation}
for all $0< r < R$.  Combining this estimate with the first equation in \eqref{sysrad2} now gives
\begin{align*}
\left[ u'(r)^{p-1-\alpha} \right]'
&\geq \frac{\delta}{n-1} \,f_1(r) g_1(v(r)) - \frac{\delta^2}{(n-1)(\delta+1)} \, f_1(r)g_1(v(r)) \\
&= \frac{p-1-\alpha}{n(p-1-\alpha) + \alpha} \,f_1(r) g_1(v(r))
\end{align*}
for all $0< r < R$.  This proves the lower bound in \eqref{est1}, and it also implies that $u'(r)$ is
increasing.  Turning to the second equation in \eqref{sysrad3}, one may then argue that
\begin{equation*}
r^{n-1} v'(r)^{p-1}
\leq f_2(r)g_2(v(r)) \cdot g_3(u'(r)) \int_0^r s^{n-1} \,ds
\end{equation*}
for all $0< r < R$, in analogy with \eqref{bd1}.  This provides an analogue of \eqref{bound} which
combines with the second equation in \eqref{sysrad2} to yield the lower bound in \eqref{est2} as before.
\end{proof}

\section{Monotonic comparison for general quasilinear systems} \label{Sec3}

In this section, we continue our study of the general system \eqref{Sys1} and we establish two important
comparison results.  Our overall plan is to compare solutions $(u,v)$ of the original system
\eqref{Sys1} with solutions $(u_0,v_0)$ of the system
\begin{equation} \label{Sys2}
\left\{ \begin{aligned}
\Delta_p u_0 &= f_1(|x|) \cdot h_1(v_0) \cdot |\nabla u_0|^{\alpha} &\quad\mbox{ in } \Omega,\\
\Delta_p v_0 &= f_2(|x|) \cdot h_2(v_0) \cdot h_3(|\nabla u_0|) &\quad\mbox{ in } \Omega,
\end{aligned} \right.
\end{equation}
which is obtained from \eqref{Sys1} by replacing the functions $g_j$ by the functions $h_j$. Intuitively
speaking, one should think of \eqref{Sys2} as a simplified version of the original system \eqref{Sys1}
whose solutions are expected to exhibit the same behavior at infinity.

When it comes to the functions $f_i$, $g_j$ and $h_j$, we impose the following assumptions:
\begin{itemize}
\item[(B1)]
each $f_i, g_j, h_j$ is continuous, non-decreasing in $[0,\infty)$ and positive in $(0,\infty)$;
\item[(B2)]
each $g_j, h_j$ is differentiable in $[0,\infty)$ and $g_1,g_2,h_1,h_2$ are increasing in
$(0,\infty)$;
\item[(B3)]
the quotients $Q_j(s,t) = g_j(st)/h_j(t)$ are non-increasing in $t$ for all $s\geq 1$, $t>0$;
\item[(B4)]
there exist constants $k_j\geq 0$ such that $\lim_{t\to\infty} Q_j(s,t) = s^{k_j}$ for each $s\geq 1$;
\item[(B5)]
one has $g_j(t) \geq h_j(t)$ for all $t>0$.
\end{itemize}

Assumption (B1) coincides with that of the previous section, but it is now also imposed on the functions
$h_j$. For the comparison results of this section, we study local solutions in the open ball $\Omega =
B_R$ without assuming (B4).  In the next section, we shall then turn to global solutions in the whole
space $\Omega = \mathbb R^n$ assuming (B1)-(B4).  It is easy to see that (B3) and (B4) trivially imply
(B5), as they imply $g_j(st) \geq s^{k_j} h_j(t)$, more generally.

\begin{lemma}[Comparison lemma] \label{ComLem}
Assume (B1), (B2), (B5), $n\geq 2$ and $\alpha\geq 0$. Suppose that $(u,v)$ and $(u_0,v_0)$ are
non-constant radial solutions of \eqref{Sys1} and \eqref{Sys2}, respectively, which are positive for all
$0<r<R$.  If $u(0) > u_0(0)$ and $v(0) > v_0(0)$, then
\begin{equation} \label{comp1}
u(r)>u_0(r), \quad v(r)>v_0(r), \quad u'(r)>u_0'(r), \quad v'(r)>v_0'(r)
\end{equation}
throughout the whole interval $(0,R)$.
\end{lemma}

\begin{proof}
Let us denote by $(0,R_0) \subset (0,R)$ the maximal subinterval on which
\begin{equation} \label{comp2}
u(r) > u_0(r), \qquad v(r) > v_0(r) \qquad\text{for all $0<r<R_0$.}
\end{equation}
The function $u$ satisfies the first equation of the system \eqref{sysrad3}, so it satisfies
\begin{equation*}
\left[ r^\delta u'(r)^{p-1-\alpha} \right]' = \frac{\delta}{n-1} \,r^\delta f_1(r) g_1(v(r)),
\end{equation*}
where $\delta>0$ is defined by \eqref{delta}, and $p-1-\alpha>0$ by Lemma \ref{Nonex}.  The function
$u_0$ satisfies the exact same equation with $g_1$ replaced by $h_1$, hence
\begin{equation*}
\left[ r^\delta (u'(r)^{p-1-\alpha} - u'_0(r)^{p-1-\alpha}) \right]' =
\frac{\delta}{n-1} \,r^\delta f_1(r) \cdot \left[g_1(v) - h_1(v_0)\right].
\end{equation*}
Along the interval $(0,R_0)$, we have $v>v_0>0$ by \eqref{comp2}, so $g_1(v) > g_1(v_0) \geq h_1(v_0)$
by (B2) and (B5).  Since the factor $f_1(r)$ is positive by (B1), the right hand side is then positive,
so it easily follows that $u'(r)>u_0'(r)$ for all $0<r<R_0$.

Next, we turn to the functions $v,v_0$.  Using the second equation in \eqref{sysrad3}, we get
\begin{equation*}
\left[r^{n-1} (v'(r)^{p-1} - v'_0(r)^{p-1})\right]' =
r^{n-1} f_2(r) \cdot \left[g_2(v) g_3(u') - h_2(v_0) h_3(u_0')\right].
\end{equation*}
When it comes to the interval $(0,R_0)$, we have $u'>u_0'$ by above, so $g_3(u') \geq g_3(u_0')$ by
(B1). In addition, $v>v_0$ by \eqref{comp2}, and thus $g_2(v) > g_2(v_0)$ by (B2).  The right hand side
is then positive by (B5), so it easily follows that $v'(r)> v_0'(r)$ for all $0<r<R_0$.

This shows that both $u(r)-u_0(r)$ and $v(r)-v_0(r)$ are increasing throughout $(0,R_0)$.  As these
functions are positive at $r=0$ by assumption, they cannot possibly vanish at $r=R_0$.  In other words,
the maximal subinterval on which \eqref{comp2} holds is the whole interval $(0,R)$, and our argument
above gives $u'(r)>u_0'(r)$ and $v'(r)>v_0'(r)$ for all $0<r<R$.
\end{proof}

\begin{theorem}[Monotonic comparison] \label{MCT}
Assume (B1)-(B3), (B5), $n\geq 2$ and $\alpha\geq 0$.  Suppose that $(u,v)$ and $(u_0,v_0)$ are
non-constant radial solutions of \eqref{Sys1} and \eqref{Sys2}, respectively, which are positive for all
$0<r<R$. Suppose also that $u(0),v(0)>0$ and
\begin{equation}\label{origin}
u_0(0) = v_0(0) = 0, \quad
\lim_{r\to 0^+} \frac{u_0(r)}{u_0'(r)} < \infty, \quad
\lim_{r\to 0^+} \frac{v_0(r)}{v_0'(r)} < \infty.
\end{equation}
Then each of the quotients
\begin{equation} \label{UVWY}
\mathcal U(r)=\frac{u(r)}{u_0(r)}, \qquad \mathcal V(r)=\frac{v(r)}{v_0(r)}, \qquad
\mathcal W(r)=\frac{u'(r)}{u'_0(r)}, \qquad \mathcal Y(r)=\frac{v'(r)}{v'_0(r)}
\end{equation}
is decreasing throughout the interval $(0,R)$.
\end{theorem}

\begin{proof}
First of all, it follows from Lemmas \ref{Mono} and \ref{Nonex} that $\alpha<p-1$ and that
\begin{equation*}
u'(r), \,v'(r), \,u_0'(r), \,v_0'(r) > 0 \quad\text{for all $0<r<R$.}
\end{equation*}
Recalling the first equation in \eqref{sysrad3}, we see that the functions $u,u_0$ satisfy the system
\begin{equation} \label{A0}
\left\{ \begin{aligned}
&\left[r^\delta u'(r)^{p-1-\alpha}\right]' = \frac{\delta}{n-1} \,r^\delta f_1(r) g_1(v(r))
&&\quad\text{for all $0<r<R$}, \\
&\left[r^\delta u_0'(r)^{p-1-\alpha}\right]' = \frac{\delta}{n-1} \,r^\delta f_1(r) h_1(v_0(r))
&&\quad\text{for all $0<r<R$}
\end{aligned} \right.
\end{equation}
with $\delta>0$ defined by \eqref{delta}.  We let $A(r) = r^\delta u_0'(r)^{p-1-\alpha}$ for simplicity
and divide the last two equations.  Using our definition \eqref{UVWY} and our assumption (B3), we arrive
at
\begin{equation*}
\frac{1}{A'(r)} \,\left[r^\delta u'(r)^{p-1-\alpha}\right]' = \frac{g_1(v)}{h_1(v_0)}
= \frac{g_1(\mathcal V v_0)}{h_1(v_0)} = Q_1(\mathcal V, v_0).
\end{equation*}
Since $u' = u_0'\mathcal W$ by our definition \eqref{UVWY}, the left hand side can also be expressed as
\begin{equation*}
\frac{1}{A'(r)} \,\left[r^\delta u_0'(r)^{p-1-\alpha} \cdot \mathcal W(r)^{p-1-\alpha}\right]' =
\frac{1}{A'(r)} \,\left[A(r) \cdot \mathcal W(r)^{p-1-\alpha}\right]'.
\end{equation*}
Once we combine the last two equations, we conclude that $\mathcal W(r)$ satisfies
\begin{equation} \label{ode1}
\frac{A(r)}{A'(r)} \cdot \left[\mathcal W(r)^{p-1-\alpha}\right]' + \mathcal W(r)^{p-1-\alpha}
= Q_1(\mathcal V(r), v_0(r)),
\end{equation}
where $A(r) = r^\delta u_0'(r)^{p-1-\alpha}$ and $Q_1$ is defined in our assumption (B3).

Repeating the exact same argument, one may derive a similar equation for $\mathcal Y(r)$, starting with
the second equation of \eqref{sysrad3} instead of the first.  This leads to the analogue
\begin{equation} \label{ode2}
\frac{B(r)}{B'(r)} \cdot \left[\mathcal Y(r)^{p-1}\right]' + \mathcal Y(r)^{p-1}
= Q_2(\mathcal V(r), v_0(r)) \cdot Q_3(\mathcal W(r), u_0'(r)),
\end{equation}
where $B(r) = r^{n-1} v_0'(r)^{p-1}$ and $Q_2, Q_3$ are defined in our assumption (B3).

We now proceed to analyze the system \eqref{ode1}-\eqref{ode2}.  Recalling \eqref{origin}, we note that
\begin{equation*}
\lim_{r\to 0^+} v_0(r) \mathcal V'(r) =
\lim_{r\to 0^+} \left[ v'(r) - \frac{v(r) v_0'(r)}{v_0(r)} \right] = - v(0) \cdot
\lim_{r\to 0^+} \frac{v_0'(r)}{v_0(r)} < 0.
\end{equation*}
Let us denote by $(0, R_1) \subset (0,R)$ the maximal subinterval on which $\mathcal V'(r) < 0$.  Along
this subinterval, the right hand side of \eqref{ode1} is easily seen to be decreasing, as $Q_1(s,t)$ is
increasing in $s$ by (B2) and non-increasing in $t$ by (B3) for all $s\geq 1$ and $t>0$.  Using the fact
that $s= \mathcal V(r) > 1$ by Lemma \ref{ComLem}, we thus have
\begin{equation*}
\frac{d}{dr} \,Q_1(\mathcal V(r), v_0(r)) = \frac{\partial Q_1}{\partial s} \cdot \mathcal V'(r) +
\frac{\partial Q_1}{\partial t} \cdot v_0'(r) < 0
\end{equation*}
for all $0<r<R_1$ because $\mathcal V'(r) < 0$ and $v_0'(r)>0$ along this interval.  Returning to
\eqref{ode1}, we conclude that the left hand side is decreasing in $r$, namely
\begin{equation} \label{arg1}
\frac{d}{dr} \left[ \frac{A(r)}{A'(r)} \cdot \left[ \mathcal W(r)^{p-1-\alpha} \right]' +
\mathcal W(r)^{p-1-\alpha} \right] < 0
\end{equation}
for all $0< r < R_1$.  Noting that $A(r)= r^\delta u_0'(r)^{p-1-\alpha}$ is positive, we now multiply
the last equation by $A(r)$ and then rearrange terms to find that
\begin{equation*}
\frac{d}{dr} \left[ A(r) \cdot \frac{A(r)}{A'(r)} \cdot
\left[ \mathcal W(r)^{p-1-\alpha} \right]' \right] < 0
\end{equation*}
for all $0< r < R_1$.  This makes the expression within the square brackets decreasing, so
\begin{equation} \label{arg2}
\frac{A(r)^2}{A'(r)} \cdot \left[\mathcal W(r)^{p-1-\alpha}\right]' <
\lim_{r\to 0^+} \frac{A(r)^2}{A'(r)} \cdot \left[\mathcal W(r)^{p-1-\alpha}\right]'
\end{equation}
for all $0<r<R_1$.  To compute the limit on the right hand side, we first note that
\begin{align*}
\lim_{r\to 0^+} \frac{A(r)^2}{A'(r)} \cdot \left[\mathcal W(r)^{p-1-\alpha}\right]'
&= \lim_{r\to 0^+} A(r) \cdot \left[ Q_1(\mathcal V(r), v_0(r)) - \mathcal W(r)^{p-1-\alpha} \right] \\
&= \lim_{r\to 0^+} \left[ r^\delta u_0'(r)^{p-1-\alpha} Q_1(\mathcal V(r), v_0(r)) -
r^\delta u'(r)^{p-1-\alpha} \right]
\end{align*}
by \eqref{ode1} and the definition \eqref{UVWY} of $\mathcal W(r)$.  Since $p-1-\alpha>0$, we thus have
\begin{align*}
\lim_{r\to 0^+} \frac{A(r)^2}{A'(r)} \cdot \left[ \mathcal W(r)^{p-1-\alpha} \right]'
&= \lim_{r\to 0^+} r^\delta u_0'(r)^{p-1-\alpha} \cdot \frac{g_1(v(r))}{h_1(v_0(r))} = 0
\end{align*}
by the definition of $Q_1$ and since our estimate \eqref{bound} ensures that
\begin{equation*}
0 \leq r^\delta u_0'(r)^{p-1-\alpha} \cdot \frac{g_1(v(r))}{h_1(v_0(r))}
\leq \frac{\delta \, r^{\delta + 1} }{(n-1)(\delta + 1)} \cdot f_1(r) g_1(v(r)).
\end{equation*}
Knowing that the limit in the right hand side of \eqref{arg2} is zero, we deduce that the left hand side
is negative.  Since $A'(r)>0$ by \eqref{A0}, this implies $\mathcal W'(r) < 0$ for all $0<r<R_1$.

Our next step is to show that $\mathcal U'(r)<0$ for all $0<r<R_1$.  Recalling \eqref{UVWY}, we get
\begin{equation} \label{arg3}
\frac{u_0(r)}{u_0'(r)} \cdot \mathcal U'(r) + \mathcal U(r) =
\frac{u'(r) u_0(r) - u_0'(r) u(r)}{u_0(r) u_0'(r)} + \frac{u(r)}{u_0(r)} = \mathcal W(r).
\end{equation}
This function is decreasing on $(0, R_1)$ by above, so it obviously satisfies
\begin{equation*}
\frac{d}{dr} \left[ \frac{u_0(r)}{u_0'(r)} \cdot \mathcal U'(r) + \mathcal U(r) \right] < 0
\end{equation*}
for all $0< r < R_1$.  Multiplying by $u_0(r)$ and rearranging terms now gives
\begin{equation} \label{arg4}
\frac{d}{dr} \left[ u_0(r) \cdot \frac{u_0(r)}{u_0'(r)} \cdot \mathcal U'(r) \right] < 0
\end{equation}
for all $0< r < R_1$.  Thus, the expression in square brackets is decreasing.  Since
\begin{equation*}
\lim_{r\to 0^+} \frac{u_0(r)^2}{u_0'(r)} \cdot \mathcal U'(r) =
\lim_{r\to 0^+} \left[ \frac{u'(r) u_0(r)}{u_0'(r)} - u(r) \right] = -u(0) < 0
\end{equation*}
by our assumption \eqref{origin}, it follows by the last two equations that
\begin{equation*}
\frac{u_0(r)^2}{u_0'(r)} \cdot \mathcal U'(r) < 0
\end{equation*}
for all $0<r<R_1$.  In particular, one also has $\mathcal U'(r) < 0$ for all $0<r<R_1$.

We note that the above analysis only uses the first equation \eqref{ode1} of our system.  Let us now
turn to the second equation \eqref{ode2}, according to which
\begin{equation*}
\frac{B(r)}{B'(r)} \cdot \left[\mathcal Y(r)^{p-1}\right]' + \mathcal Y(r)^{p-1}
= Q_2(\mathcal V(r), v_0(r)) \cdot Q_3(\mathcal W(r), u_0'(r)).
\end{equation*}
Knowing that $\mathcal V(r)$ and $\mathcal W(r)$ are decreasing in $(0, R_1)$, one may easily check that
the right hand side is itself decreasing because of our assumptions (B2)-(B3) and since $u_0''(r)>0$ by
Lemma \ref{Conv}. Thus, the left hand side is decreasing as well, and we get
\begin{equation*}
\frac{d}{dr} \left[ \frac{B(r)}{B'(r)} \cdot \left[\mathcal Y(r)^{p-1}\right]' +
\mathcal Y(r)^{p-1} \right] < 0
\end{equation*}
for all $0<r<R_1$, an exact analogue of \eqref{arg1}.  Proceeding as before, one may then multiply by
the positive factor $B(r)$ and integrate to deduce that $\mathcal Y'(r) < 0$ for all $0<r<R_1$.

As our final step, we now relate $\mathcal Y(r)$ with $\mathcal V(r)$.  According to \eqref{UVWY}, we
have
\begin{equation*}
\frac{v_0(r)}{v_0'(r)} \cdot \mathcal V'(r) + \mathcal V(r) =
\frac{v'(r) v_0(r) - v_0'(r) v(r)}{v_0'(r) v_0(r)} + \frac{v(r)}{v_0(r)} = \mathcal Y(r),
\end{equation*}
and this provides an exact analogue of \eqref{arg3}.  Since the right hand side is decreasing, the same
is true for the left hand side, so we obtain an exact analogue of \eqref{arg4}, namely
\begin{equation*}
\frac{d}{dr} \left[ v_0(r) \cdot \frac{v_0(r)}{v_0'(r)} \cdot \mathcal V'(r) \right] < 0
\end{equation*}
for all $0< r< R_1$.  On the other hand, $(0, R_1)$ was assumed to be the maximal subinterval on which
$\mathcal V'(r)$ is negative.  Along this interval, the expression in square brackets is decreasing and
negative, so it cannot possibly vanish at $r=R_1$.  This implies that $(0, R_1)$ must be the whole
interval $(0,R)$. In particular, the functions $\mathcal U(r)$, $\mathcal V(r)$, $\mathcal W(r)$,
$\mathcal Y(r)$ are all decreasing throughout the whole interval $(0,R)$, and the proof is complete.
\end{proof}

\section{Asymptotic comparison for general quasilinear systems} \label{Sec4}

\begin{theorem}[Asymptotic Comparison] \label{ACT}
Assume (B1)-(B4), $n\geq 2$ and $\alpha\geq 0$. Suppose that $(u,v)$ and $(u_0,v_0)$ are non-constant
radial solutions of \eqref{Sys1} and \eqref{Sys2}, respectively, which are positive for all $r>0$.  If
$u(0),v(0)>0$ and $(u_0,v_0)$ satisfies \eqref{origin}, then
\begin{equation*}
\lim_{r\to \infty} \frac{u(r)}{u_0(r)} = \lim_{r\to \infty} \frac{v(r)}{v_0(r)} =
\lim_{r\to \infty} \frac{u'(r)}{u_0'(r)} = \lim_{r\to \infty} \frac{v'(r)}{v_0'(r)} = 1.
\end{equation*}
\end{theorem}

\begin{proof}
Since \eqref{origin} holds, one has $u(0)>0=u_0(0)$ and $v(0) > 0 = v_0(0)$, so Theorem \ref{MCT} is
applicable.  Thus, the functions $\mathcal U(r), \mathcal V(r), \mathcal W(r), \mathcal Y(r)$ defined by
\eqref{UVWY} are decreasing for all $r>0$.  Using Lemma \ref{Mono} and the fact that $\mathcal U(r)$ is
decreasing, we get
\begin{equation} \label{U>W}
0 > \mathcal U'(r) \cdot \frac{u_0(r)}{u_0'(r)} =
\frac{u'(r) u_0(r) - u_0'(r) u(r)}{u_0(r) u_0'(r)} = \mathcal W(r) - \mathcal U(r),
\end{equation}
and thus $\mathcal W(r) < \mathcal U(r)$ for all $r>0$.  Similarly, one finds that $\mathcal Y(r) <
\mathcal V(r)$ for all $r>0$.

Let us now recall the first equation \eqref{ode1} of our system, according to which
\begin{equation*}
\frac{A(r)}{A'(r)} \cdot \left[\mathcal W(r)^{p-1-\alpha}\right]' + \mathcal W(r)^{p-1-\alpha}
= Q_1(\mathcal V(r), v_0(r))
\end{equation*}
with $A(r) = r^\delta u_0'(r)^{p-1-\alpha}$ and $Q_1$ as in our assumption (B3).  Note that $A'(r)>0$ by
\eqref{A0} and that $p-1-\alpha>0$ by Lemma \ref{Nonex}.  Since $\mathcal W(r)$ is decreasing, it
follows by the last equation and our assumptions (B3)-(B4) that
\begin{equation*}
\mathcal W(r)^{p-1-\alpha} > Q_1(\mathcal V(r), v_0(r))
\geq \mathcal V(r)^{k_1} \geq \mathcal Y(r)^{k_1}
\end{equation*}
for all $r>0$.  Applying the same argument to the second equation \eqref{ode2}, we get
\begin{align*}
\mathcal Y(r)^{p-1}
&> Q_2(\mathcal V(r), v_0(r)) \cdot Q_3(\mathcal W(r), u_0'(r)) \\
&\geq \mathcal V(r)^{k_2} \cdot \mathcal W(r)^{k_3}
\geq \mathcal Y(r)^{k_2} \cdot \mathcal W(r)^{k_3}
\end{align*}
for all $r>0$.  In view of the last two estimates, one must then have
\begin{equation*}
\mathcal Y(r)^{(p-1-k_2)(p-1-\alpha)} > \mathcal W(r)^{k_3(p-1-\alpha)} \geq \mathcal Y(r)^{k_1 k_3}
\end{equation*}
for all $r>0$.  Since $\mathcal Y(r)>1$ by Lemma \ref{ComLem}, however, this is only possible when
\begin{equation} \label{cond}
(p-1-k_2)(p-1-\alpha) > k_1 k_3.
\end{equation}

Next, we analyze the behavior of solutions as $r\to \infty$.  The functions $\mathcal U(r)$, $\mathcal
V(r)$, $\mathcal W(r)$ and $\mathcal Y(r)$ are all decreasing and positive, so they all attain a limit
as $r\to \infty$.  Let us denote their limits by $\mathcal U_\infty$, $\mathcal V_\infty$, $\mathcal
W_\infty$ and $\mathcal Y_\infty$, respectively.  We must then have
\begin{equation} \label{LHR}
\mathcal U_\infty = \lim_{r\to\infty} \frac{u(r)}{u_0(r)} =
\lim_{r\to\infty} \frac{u'(r)}{u_0'(r)} = \mathcal W_\infty
\end{equation}
by L'H\^opital's rule, and similarly, $\mathcal V_\infty = \mathcal Y_\infty$.  According to the system
\eqref{A0}, one has
\begin{equation*}
\left\{ \begin{aligned}
&(p-1-\alpha) \cdot u'(r)^{p-2-\alpha} u''(r) =
\frac{\delta}{n-1} \, f_1(r) g_1(v(r)) - \frac{\delta}{r} \, u'(r)^{p-1-\alpha}, \\
&(p-1-\alpha) \cdot u_0'(r)^{p-2-\alpha} u_0''(r) =
\frac{\delta}{n-1} \, f_1(r) h_1(v_0(r)) - \frac{\delta}{r} \, u_0'(r)^{p-1-\alpha}.
\end{aligned} \right.
\end{equation*}
Dividing these two equations and recalling our definition \eqref{UVWY}, one can thus write
\begin{equation*}
\mathcal W(r)^{p-2-\alpha} \,\frac{u''(r)}{u_0''(r)} = \frac{\frac{\delta}{n-1} \,
f_1(r) g_1(v(r)) \cdot u_0'(r)^{\alpha+1-p} - \frac{\delta}{r} \mathcal W(r)^{p-1-\alpha}}
{\frac{\delta}{n-1} \, f_1(r) h_1(v_0(r)) \cdot u_0'(r)^{\alpha+1-p} - \frac{\delta}{r}}.
\end{equation*}
Taking the limit as $r\to \infty$ now leads to the identity
\begin{equation} \label{iden}
\mathcal W_\infty^{p-2-\alpha} \,\lim_{r\to \infty} \frac{u''(r)}{u_0''(r)} =
\lim_{r\to \infty} \frac{g_1(v(r))}{h_1(v_0(r))} =
\lim_{r\to \infty} Q_1(\mathcal V(r), v_0(r))
\end{equation}
with $Q_1$ as in our assumption (B3).  The limit on the right hand side exists because
\begin{equation*}
\lim_{r\to \infty} Q_1(\mathcal V(r), v_0(r)) =
\lim_{r\to \infty} Q_1(\mathcal V_\infty, v_0(r)) = \mathcal V_\infty^{k_1}
\end{equation*}
by our assumption (B4).  Thus, the limit on the left hand side of \eqref{iden} also exists. Using this
fact along with L'H\^opital's rule and \eqref{LHR}, we see that \eqref{iden} reduces to
\begin{equation} \label{lim1}
\mathcal U_\infty^{p-1-\alpha} = \mathcal V_\infty^{k_1}.
\end{equation}
This condition was derived by combining the equations \eqref{A0} satisfied by $u(r)$ and $u_0(r)$.  The
exact same argument applies to $v(r)$ and $v_0(r)$, so one similarly has
\begin{equation*}
\mathcal Y_\infty^{p-2} \,\lim_{r\to \infty} \frac{v''(r)}{v_0''(r)} =
\lim_{r\to \infty} Q_2(\mathcal V(r), v_0(r)) \cdot Q_3(\mathcal W(r), u_0'(r))
\end{equation*}
in analogy with \eqref{iden}.  Since $\mathcal Y_\infty = \mathcal V_\infty$ by above, our previous
approach then gives
\begin{equation} \label{lim2}
\mathcal V_\infty^{p-1} =
\mathcal V_\infty^{k_2} \cdot \mathcal W_\infty^{k_3} =
\mathcal V_\infty^{k_2} \cdot \mathcal U_\infty^{k_3}
\end{equation}
in analogy with \eqref{lim1}.  Once we combine \eqref{lim2} with \eqref{lim1}, we find that
\begin{equation*}
\mathcal V_\infty^{(p-1-k_2)(p-1-\alpha)} =
\mathcal U_\infty^{k_3(p-1-\alpha)} = \mathcal V_\infty^{k_1 k_3},
\end{equation*}
where $\mathcal V_\infty\geq 1$ by Lemma \ref{ComLem}.  In view of \eqref{cond}, the only possibility is
then $\mathcal V_\infty = 1$, which also implies that $\mathcal U_\infty = 1$.  This is precisely the
assertion of the theorem.
\end{proof}

\section{Proof of Theorem \ref{MainThm}} \label{Sec5}

First of all, we seek solutions of the associated system \eqref{sys|2} that have the form
\begin{equation} \label{uv}
u(r) = C_\lambda r^\lambda, \qquad v(r) = C_\mu r^\mu.
\end{equation}
In view of the definition of the $p$-Laplace operator, it is easy to check that
\begin{equation} \label{expl}
\Delta_p (C_\lambda r^\lambda) = (\lambda C_\lambda)^{p-1} \,(\lambda (p-1) + n-p)
\cdot r^{(p-1) \lambda - p}.
\end{equation}
Inserting this identity in \eqref{sys|2} and comparing powers of $r$, one obtains the system
\begin{align*}
(p-1 -\alpha) \lambda - k_1 \mu &= p - \alpha + m_1, \\
(p-1 - k_2) \mu - k_3 \lambda &= p -k_3 + m_2.
\end{align*}
Since \eqref{excon} holds, this system has a unique solution $(\lambda, \mu)$ which is given by
\begin{align*}
\lambda &= \frac{(p -\alpha + m_1)(p-1-k_2) + k_1(p -k_3 +m_2)}{(p-1-\alpha)(p-1-k_2) -k_1 k_3} > 1, \\
\mu &= \frac{(p-1-\alpha)(p +m_2) + k_3(1 + m_1)}{(p-1-\alpha)(p-1-k_2) - k_1 k_3} > 1.
\end{align*}
Using \eqref{expl} once again, we find that \eqref{uv} is a solution of \eqref{sys|2}, if and only if
\begin{align*}
(\lambda C_\lambda)^{p-1} \cdot (\lambda (p-1) + n-p)
&= c_1 C_\mu^{k_1} \cdot (\lambda C_\lambda)^\alpha, \\
(\mu C_\mu)^{p-1} \cdot (\mu (p-1) + n-p) &= c_2 C_\mu^{k_2} \cdot (\lambda C_\lambda)^{k_3}.
\end{align*}
Let us set $B_\lambda = \lambda (p-1) + n-p$ for simplicity.  Then we get the equivalent system
\begin{align*}
(\lambda C_\lambda)^{p-1-\alpha} = \frac{c_1}{\mu^{k_1} B_\lambda} \,(\mu C_\mu)^{k_1}, \qquad
(\mu C_\mu)^{p-1-k_2} = \frac{c_2}{\mu^{k_2} B_\mu} \,(\lambda C_\lambda)^{k_3}.
\end{align*}
Since \eqref{excon} holds, this system has a unique solution $C_\lambda, C_\mu$ which is given by
\begin{align*}
(\lambda C_\lambda)^{(p-1-\alpha)(p-1-k_2)-k_1k_3} &=
\left( \frac{c_1}{\mu^{k_1} B_\lambda} \right)^{p-1-k_2}
\left( \frac{c_2}{\mu^{k_2} B_\mu} \right)^{k_1}, \\
(\mu C_\mu)^{(p-1-\alpha)(p-1-k_2)-k_1k_3} &=
\left( \frac{c_2}{\mu^{k_2} B_\mu} \right)^{p-1-\alpha}
\left( \frac{c_1}{\mu^{k_1} B_\lambda} \right)^{k_3}.
\end{align*}

The exact values of $\lambda, \mu, C_\lambda, C_\mu$ do not play any role in our approach, as we shall
simply resort to our general Theorems \ref{MCT} and \ref{ACT} which apply for arbitrary solutions $(u_0,
v_0)$ that vanish at the origin.  In this case, our solution $(u_0,v_0)$ is given by \eqref{uv}, so
\begin{equation*}
\lim_{r\to 0^+} \frac{u_0(r)}{u_0'(r)} = \lim_{r\to 0^+} \frac{r}{\lambda} = 0, \qquad
\lim_{r\to 0^+} \frac{v_0(r)}{v_0'(r)} = \lim_{r\to 0^+} \frac{r}{\mu} = 0
\end{equation*}
and condition \eqref{origin} holds.  Let us then resort to Theorem \ref{MCT} for the special case
\begin{equation*}
f_i(r) = c_i r^{m_i}, \qquad h_j(s) = s^{k_j},
\end{equation*}
where $i=1,2$ and $j=1,2,3$.  Our assumptions (B1)-(B2) hold trivially, while
\begin{equation*}
Q_j(s,t) = \frac{g_j(st)}{h_j(t)} = \frac{g_j(st)}{(st)^{k_j}} \cdot s^{k_j}
\end{equation*}
is non-increasing in $t$ because of (A3).  It also follows by condition \eqref{limit} that
\begin{equation*}
\lim_{t\to \infty} Q_j(s,t) = \lim_{t\to \infty} \frac{g_j(st)}{(st)^{k_j}} \cdot s^{k_j} = s^{k_j}
\end{equation*}
for each $s>0$, so our assumptions (B1)-(B5) are all valid.  Thus, part (a) of Theorem \ref{MainThm}
follows by Theorem \ref{MCT}, while part (b) of Theorem \ref{MainThm} follows by Theorem \ref{ACT}.

\section*{Acknowledgements}
The first author acknowledges the financial support of The Irish Research Council Postgraduate
Scholarship under grant number GOIPG/2022/469.

\end{document}